\documentclass[10pt,leqno]{amsart}
\usepackage{graphicx}
\usepackage{amssymb,amsthm,amsmath}
\usepackage{hyperref}
\newtheorem{theorem}{Theorem}[]
\newtheorem{proposition}[theorem]{Proposition}

\newtheorem{remark}[theorem]{Remark}

\hypersetup{ colorlinks=true, linkcolor=black, filecolor=black, urlcolor=black }

\begin{document}
\title{On one of Birkhoff's theorems for backward limit points} 
\author[V. R{\'{y}}{\v{z}}ov{\'{a}}]{Veronika R{\'{y}}{\v{z}}ov{\'{a}}}
\date{}
\address{Mathematical Institute, Silesian University in Opava, Na Rybn{\'{i}}{\v{c}}ku 1, 74601, Opava, The Czech republic}
\email{veronika.ryzova@math.slu.cz}

\let\thefootnote\relax
\footnotetext{MSC2020: Primary 37B20, Secondary 37E05.}

\begin{abstract}
In 1927 George Birkhoff in his book Dynamical Systems presented a~theorem that describes the~behaviour of trajectories outside of a~set of non-wandering points on an~arbitrary compacta. Much later in 1960s Sharkovsky followed up on Birkhoff's work and published even stronger result, this time focusing on the~set of omega limit points for interval maps. In this article we formulate similar statement for a~neighbourhood of a~set of different types of backward limit points for maps of the~interval.
 
\end{abstract}

\bigskip
\maketitle
\section{Preliminaries}

Before we immerse ourselves in theorems and definitions, let us state a few historical remarks. The Birkhoff's conclusion was firstly presented in \cite{bir} for general class of continuous dynamical systems. As such the terms curve of motion, wandering motion or recurrent motion are used frequently throughout the whole section dedicated to this problem. In case of discrete dynamical system they can be understood simply as wandering (resp. recurrent) point. Let us now reveal the theorem in its original form \cite{bir}:

\noindent \textit{Any wandering motion remains outside of a prescribed neighbourhood of} $M_1$ (i.e. the set of non-wandering points) \textit{only a finite time T, and goes out of this neighbourhood only a finite number of times N, where N and T are uniformly limited, once the neighbourhood is chosen.} 

Our main result proofs the reformulation of this theorem for non-invertible discrete dynamical systems and sets of backward limit points.

\subsection{Basic Definitions}
We define a~discrete \textit{dynamical system} as an~ordered pair $\left(X, f\right)$, where $X$ is a~compact metric space and $f: X \rightarrow X$ is a~continuous map. We say that a~set $A \subset X$ is \textit{invariant} if $f\left(A\right) \subset A$, in case of equality we use the~term \textit{strongly invariant}. The dynamical system $\left(X, f\right)$ is called \textit{minimal} if and only if it does not contain any non-empty, proper, closed invariant subset.

For any positive integer $n$ and a~point $x \in X$ we denote by $f^n\left(x\right)$ the~$n$\textit{-th iteration} of the~map $f$, which is to be understood as follows $f^0 \left(x\right) = x$ and $f^n\left(x\right) = f\left(f^{n-1}\left(x\right)\right)$. A~sequence $\left\{x_n\right\}_{n = 0}^{\infty}$ obeying the~rule $f^n\left(x\right) = x_n$ is called the~\textit{trajectory of a~point} $x$. The~notation $Orb_f\left(x\right) = \left\{f^n\left(x\right): n\geq 0\right\}$ is used for a~subset of $X$ known as a~\textit{forward orbit of a~point} $x.$ 

Throughout the~article we are examining backward dynamics of our system, some of the~notions already defined have their backward counterparts. Whenever such situation arises and such concept will be relevant to our problem we will provide the~interpretation of the~term in a~point of view of backward dynamics as well. Switching to backward notions is fairly transparent in case of  a~\textit{backward orbit of a~point} $x$, we simply put $Orb_{f}^{-}\left(x\right) = \left\{f^{-n}\left(x\right): n\geq 0\right\}.$ 

However some difficulties arises during the~construction of a~backward trajectory of a~point $x.$ These issues are addressed by forming two different types of sequences. The~first one is called a~\textit{preimage sequence of a~point} $x$. It is a~sequence $\left\{x_n\right\}_{n = 0}^{\infty}$ such that $f^n\left(x_n\right) = x,$ as you can see there is no prescribed relation between two consecutive members. The~backward dynamics of a~point can be very rich. This phenomenon manifests itself by the~existence of many preimage sequences. To bring some order into their distinction a~\textit{backward orbit branch of a~point} $x$ was defined. It is once more a~sequence $\left\{x_n\right\}_{n = 0}^{\infty}$, additionally as opposed to the~preimage sequence there is a~condition for the~adjacent members of the~sequence, namely $f\left(x_n\right) = x_{n-1}.$

Forward trajectory of a~point $x$ can display various patterns, some of the~most well known are: \textit{fixed point} (a~point $x$ is a~fixed point if $f^n\left(x\right) = x$ for every $n \in \mathbb{N}),$ \textit{periodic point} ($x$ is a~periodic point if there exists an~integer $k$ such that $f^k\left(x\right) = x$ and for every $n < k$ holds $f^n\left(x\right) \neq x,$ said $k$ is then called the~\textit{smallest period} of $x$), \textit{recurrent point} (we call $x$ recurrent if for every neighbourhood $U$ of $x$ there exists a~positive integer $n$ such that $f^n\left(x\right) \in U$) and \textit{non-wandering point} ($x$ is non-wandering if for every neighbourhood $U$ of $x$ there exists a~positive integer $n$ such that $f^n\left(U\right) \cap U \neq \emptyset$). Sets of these points will be denoted by: $Fix\left(f\right),\,Per\left(f\right),\,Rec\left(f\right)$ and $NW\left(f\right)$, respectfully.

Trajectory of every point contains in itself some information about the~dynamical properties of the~system, to analyse the long term behaviour of the~system further one can examine their limit sets. 
\smallskip
\subsection{Forward and Backward Limit Sets}
The~set of limit points of forward trajectory of a~point $x$, widely known as an~$\omega-$limit set of $x$ (or $\omega\left(x\right)$ for short), is a~topic studied in detail by many authors. Every such set is non-empty, compact and strongly invariant. It is known whether given a~closed invariant subset of $X$ is an~$\omega$-limit set of some point for continuous map of the~interval, the~criterion can be found in \cite{Blokh1996THESO}. 

From the plethora of results concerning $\omega\left(x\right),$ we point out a few widely known ones for selfmap of the interval. For topological characterisation of $\omega\left(x\right)$ one can take a~look into the~following articles \cite{ABCP} or \cite{Shar}, in the~latter Sharkovsky constructed the~decomposition of the~set $\bigcup\limits_{x \in X}\omega\left(x\right) = \omega\left(f\right)$ into sets of first type (lately known as solenoidal set), second type (also known as basic set) and the~cycle (or periodic orbit), proving that each $\omega\left(x\right)$ of a~continuous map of the~interval is contained in maximal one, denoted $\tilde{\omega}.$ Since then other approaches to the~decomposition of limit sets of piecewise monotone continuous maps were published, for example \cite{nitecki1982topological} or \cite{Blokh_1995}, where the~method is presented together with its correspondence to Sharkovsky's results. 

The notion of backward limit set was defined as a~dual concept to the~$\omega$-limit set and as such the~task of finding the~backward limit set for homeomorphisms can be simplified to finding $\omega$-limit set for $f^{-1}.$ Furthermore in case of invertible maps the terms preimage sequence and backward orbit branch are one and the~same, meaning that their limit sets also coincides. Therefore for our purposes noninvertible maps are much more significant. While looking for backward limit set of noninvertible map we can alternate between computing the~limit points of preimage sequences as in \cite{cv} to checking limit points of every backward orbit branch \cite{hero, zakl} or we can even take into consideration only one backward orbit branch \cite{ba, zaklzpv}. 

The~set consisting of limit points of all preimage sequences of a~point $x$ is called an~$\alpha$\textit{-limit set} abbreviated by $\alpha\left(x\right)$. If we compute the~limit points of all backward orbit branches of a~point $x$, we found so call \textit{special} $\alpha$\textit{-limit set} (or $s\alpha\left(x\right)$ for short).  Lastly we define an~\textit{$\alpha$-limit set of a~backward orbit branch $\left\{x_j\right\}_{j \leq 0}$.} It consists of all points $y$ for which there exists a~strictly decreasing sequence of negative integers $\left\{j_i\right\}_{i \geq 0}$ such that $x_{j_{i}} \rightarrow y$ as $i \rightarrow \infty$, denoted by $\alpha\left(\left\{x_j\right\}_{j \leq 0}\right).$ The $s\alpha\left(x\right)$ could be described in another way by using the $\alpha$-limit set of a~backward orbit branch. It is, quite simply, the union of $\alpha\left(\left\{x_j\right\}_{j \leq 0}\right)$ over all backward orbit branches of the point $x.$

In a light of the discussion above, it is clear that for every fixed backward orbit branch $\left\{x_j\right\}_{j \leq 0}$ the following chain of inclusions holds  $\alpha\left(\left\{x_j\right\}_{j \leq 0}\right) \subset s\alpha\left(x\right) \subset \alpha\left(x\right).$ Moreover the~relation between $\bigcup_{x \in X} s\alpha\left(x\right) = SA\left(f\right)$ and $Rec\left(f\right)$ was discovered in \cite{salphaC}, and from \cite{cv} we know that the~set $\bigcup_{x \in X} \alpha\left(x\right) = A\left(f\right)$ includes every non-wandering point (the~other inclusion does not hold, see an example below).

 Let us consider the~map $f: \left[0, 1\right] \rightarrow \left[0, 1\right]$ as depicted below. The backward trajectory of $x=\frac{1}{4}$ consists only of one backward orbit branch $\left\{4^{-\left(n+2\right)}\right\}_{n=0}^{\infty}.$ Due to the~uncomplicated behaviour of this sequence only one point $0$ belongs to $\alpha\left(x\right).$ However more important is the~observation, that  $x = \frac{1}{4}$ is a~wandering point that does not belong to its own $\alpha\left(x\right).$ Moreover if we study $\alpha\left(1\right)$ we will come to the~conclusion that our point $x$ is undoubtedly its part. From this it is easy to see that $NW\left(f\right) \not\supset A\left(f\right).$

\begin{figure}[h]
    \centering{
    \includegraphics[width=0.3\textwidth]{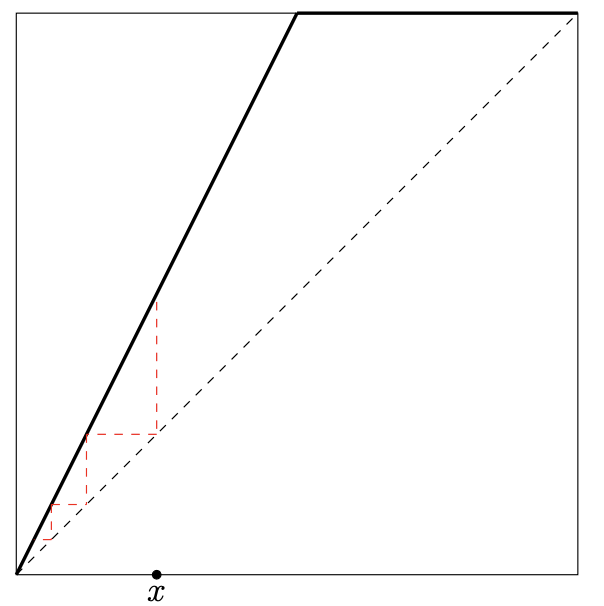}
    \centering{\caption{Graph of the~map $f: \left[0, 1\right] \rightarrow \left[0, 1\right].$}}}
      
\end{figure}

The sets defined above are in general settings of compact metric space related in the following fashion:  

$$Fix\left(f\right) \subset Per\left(f\right)  \subset Rec\left(f\right) \subset SA\left(f\right) \subset \overline{Rec\left(f\right)} \subset \omega\left(f\right) \subset NW\left(f\right) \subset A\left(f\right)$$

For the rest of the article we will restrict ourselves to the selfmap of the interval, this assumption significantly simplifies the summary above, because $\overline{Per\left(f\right)} = \overline{Rec\left(f\right)}.$

\section{Known Results}

It was already mentioned that each $\omega\left(x\right)$ for continuous map of the~interval is contained in a~maximal one, this is not true for $s\alpha\left(x\right)$ (a~counterexample was constructed in \cite{zakl}), however there is a~way to connect these notions. Let us first recall the~types of $\tilde{\omega}$ together with their properties. We will start with a~periodic orbit, this $\tilde{\omega}$ has the~least complicated structure. It is a~finite set whose elements map onto each other. The~following theorem can be found in \cite{KMS}, interestingly enough it holds for all periodic orbits, even for those that are not maximal.

\begin{theorem} \cite{KMS}
Let $P$ be a periodic orbit for the continuous interval map. If $\alpha\left(x\right) \cap P \not = \emptyset,$ then $s\alpha\left(x\right) \supset P.$
\end{theorem}

The~other two types of $\tilde{\omega}$ are infinite sets. First, we will concentrate on a~basic set. This kind of maximal omega limit set is contained in a~minimal cycle of intervals, it always includes a~periodic point and the~existence of basic set is equivalent with positive topological entropy \cite{tentri}. The relation of $s\alpha\left(x\right)$ and basic set (below) first appeared in \cite{zakl}. 

\begin{theorem}\cite{zakl}
\label{basic} For every basic set $\omega$ there exists $x \in \omega$ such that $s\alpha\left(x\right) \supset \omega.$
 \end{theorem}

Lastly let us focus on the~final type of maximal omega limit set, so called solenoidal set. Similarly as basic sets, solenoidal sets are contained in a~nested sequence of cycles intervals with periods tending to infinity, however as oppose to the~basic sets, they do not include any periodic point. At least two different ways of describing these sets can be found. According to \cite{solenoid} every solenoidal set can be written as a~union of two sets $\tilde{\omega} = P \cup Q,$ where $Q$ is a~Cantor set and $P$ is either empty or countable set of isolated points disjoint with $Q.$ Blokh in \cite{Blokh1} defines strongly solenoidal set as an~invariant subset $\Omega \subset \bigcap_{n \geq 0}M_n,$ where  $\left\{M_n\right\}_{n=0}^{\infty}$ is a~family of invariant compacta. He follows the~definition with a~theorem connecting this notion with maximal omega limit sets.

\begin{theorem}\cite{Blokh1}
Let $S = \bigcap_{n\geq0} M_n$ be a~strongly solenoidal set. There exists a~minimal set $\Omega^{\prime} \subset S,$ a maximal (with respect to inclusion) set $\omega\left(z\right) = \Omega^{\prime\prime} \subset S$ among the~$\omega$-limit sets, and a~set $\Omega^{\prime\prime\prime} = S \cap \Omega\left(f\right)$ such that:
\begin{enumerate}
\item for each $y \in S$ we have $\omega \left(y\right) = \Omega^{\prime} = S \cap \overline{Per\left(f\right)},\, \Omega^{\prime\prime\prime}\setminus \Omega^{\prime}$ consists of isolated points, and $\Omega^{\prime\prime}\setminus\Omega^{\prime}$ is either empty or countable
\item for each $y^{\prime}$ $\omega\left(y^{\prime}\right) \cap S \not = \emptyset$ implies $\Omega^{\prime} \subset \omega\left(y^{\prime}\right) \subset \Omega^{\prime\prime} \subset \Omega^{\prime\prime\prime}.$
\end{enumerate}
\end{theorem}

To summarise the~solenoidal set is a~type of maximal omega limit set, that either consists of only minimal set, i.e. $\tilde{\omega} = \Omega^{\prime} = Q$ (set $P$ is empty) or in a~very specific cases we have to add at most countably many isolated points, such that each interval contiguous to $Q$ contains at most two points of $P$, i.e.  $\tilde{\omega} = \Omega^{\prime\prime} = Q \cup P.$ The relation between a~solenoidal set and a~backward limit set was published in \cite{zakl}.

\begin{theorem}\cite{zakl}
Let $Orb\left(I_0\right) \supset Orb\left(I_1\right)\supset \ldots$ be a~nested sequence of cycles of intervals for the~continuous interval map $f$ with periods tending to infinity. Let $Q = \bigcap Orb\left(I_n\right)$ and $S = Q \cap Rec\left(f\right).$
\begin{enumerate}
\item If $ \alpha\left(x\right) \cap Q \not = \emptyset,$ then $x \in Q.$
\item If $x \in Q,$ then $s\alpha\left(x\right) \supset S$ and $s\alpha\left(x\right) \cap Q= S.$
\end{enumerate}
\end{theorem}
The main result deals with the~behaviour of backward orbit branches outside of a~set of different types of backward limit points. The~inspiration for stating such theorem comes from the~work of Birkhoff, who examined the~forward trajectories outside of a~set of non-wandering points on an~arbitrary compacta $E$.

\begin{theorem}\cite{bir} For an~arbitrary neighbourhood $U$ of $NW\left(f\right)$ there exists an~$n > 0$ such that the~number of iterates $\left\{f^j\left(x\right)\right\}_{j=0}^{\infty},\,\,x \in E$ that are outside of $U$ is not larger than $n.$
 \end{theorem}

 \noindent In case that $E$ is a~closed interval even stronger result holds.

 \begin{theorem} \cite{sahr}
 \label{omega}For an~arbitrary neighbourhood $U$ of $\omega\left(f\right)$ there exists an~$n > 0$ such that the~number of iterates $\left\{f^j\left(x\right)\right\}_{j=0}^{\infty},\,\,x \in E$ that are outside of $U$ is not larger than $n.$
 \end{theorem}
 
 The following statement is an~important part of the~proof of Theorem \ref{omega} as well as one of the~main results.
 \begin{proposition}  \cite{blc} A~point $c \in I$ lies in $\omega\left(f\right)$ if every open interval with left (resp. right) endpoint $c$ contains at least two points of some trajectory.

 \end{proposition}
 \begin{remark} The~result of the~proposition above will not change if we replace trajectory with backward orbit branch.
 \end{remark}

\section{Main Results}
In this section we formally restate Birkhoff's theorem for sets of alpha and special alpha limit points for non-invertible systems. The theorem is valid entirely in case of alpha limit points. Unfortunately, to ensure that it holds even if we take into consideration the set of special alpha limit points we need to add an assumption on the map. The idea of the proof of the following theorem comes from \cite{sahr}. 
\begin{theorem}
\label{alpha}
Let $I$ be a~compact interval, $f:I \rightarrow I$ a~continuous onto mapping. For an~arbitrary neighbourhood $U$ of $A\left(f\right)$ there exists $M \in \mathbb{N}$ such that at most $M$ points of any backward orbit branch $\left\{x_i\right\}_{i \leq 0},\,\, x_i \in I$  lie outside of $U$.
\end{theorem}

\begin{proof} Without the loss of generality we can assume that $x'\in I \setminus U$ (otherwise we simply choose the first point of backward orbit branch outside of $U$). We claim that for every such $x'$ we can find an~open neighbourhood $U_{x'}$ with at most two points of any backward orbit branch (otherwise such point will belong to $\omega\left(f\right)$ and since $\omega\left(f\right) \subset A\left(f\right)$ it will be a~part of $A\left(f\right)$ as well).

\noindent The~union of $U_{x'}$ over all $x'\in I \setminus U$ is an~open cover of $I \setminus U.$ Because $I \setminus U$ is compact we can find its finite subcover. Let us denote this finite subcover by $V_1, V_2, \ldots, V_m.$
 \smallskip

 \noindent Each $V_i,\,\,i \in \left\{1, \ldots, m\right\}$ contains at most two points of any backward orbit branch. Then $2 \cdot m =: M$ is the~maximum number of points of any backward orbit branch that lie outside of $U.$
\end{proof}

\begin{theorem} \cite{zakl} Let $I$ be a~closed interval, $f:I \rightarrow I$ a~continuous onto mapping. If $NW\left(f\right) = I$ then $s\alpha\left(x\right) = \alpha\left(x\right)$ for every $x \in I.$
\end{theorem}

\begin{remark}
 If we assume that the~condition of the~Theorem above is fulfilled, Theorem \ref{alpha} is valid for the~set $SA\left(f\right)$ as well.   
\end{remark}
 
\smallskip

\begin{theorem} Let $I$ be a~compact interval, $f:I \rightarrow I$ a~continuous onto mapping with $Rec\left(f\right) closed$ (i.e. solenoidal set does not contain any isolated points). For an~arbitrary neighbourhood $U$ of $SA\left(f\right)$ there exists $M \in \mathbb{N}$ such that at most $M$ points of any backward orbit branch $\left\{x_i\right\}_{i \leq 0},\,\, x_i \in I$  lie outside of $U$.
\end{theorem}

\begin{proof} The~Theorem holds for union of $\omega$-limit sets (Theorem \ref{omega}). From this we can conclude that only the~points that are a~part of some $\omega\left(x\right)$ but are not included in any $s\alpha\left(x\right)$ can be problematic.

 \noindent Let us assume that there is a~point $z$ such that $z \not \in s\alpha\left(x\right)$ for any $x \in I$ however it is an~omega limit point. It has to be a~part of some maximal omega limit set $\Tilde{\omega}.$ We will focus on each type of $\Tilde{\omega}$ separately.

\noindent The~easiest is the~case of periodic orbit, because then $z$ has to belong to its own $s\alpha\left(z\right)$ and as such is a~part of $SA\left(f\right).$

\noindent In another case scenario, when $\Tilde{\omega}$ is a~basic set, Theorem \ref{basic} can be applied. From this we can conclude that $z$ is in fact a~part of $SA\left(f\right).$

\noindent This leaves us with with the~last type of $\Tilde{\omega}$ namely solenoidal set. According to our assumption we are dealing with minimal system (the~restriction of $f$ to the~perfect set $Q$). In such settings every point $z$ is a~part of its own $s\alpha\left(z\right)$, meaning that every point of $\tilde{\omega}$ is a~part of $SA\left(f\right)$ as well.
\end{proof}

\begin{remark}
 If we take into consideration the~solenoidal sets that have nonempty $P$ the~situation becomes more complicated. The isolated point $y \in P$ has to be a~part of $A\left(f\right)$ and as such is a~limit point of a~subsequence of backward trajectory. It is not possible to find a~neighbourhood of $y$ that includes only a~finite amount of elements of any backward orbit branch, which makes the~theorem not valid.   
\end{remark}

\section{Acknowledgements}

I would like to express sincere gratitude to my supervisor Michal Málek for his insightful comments and suggestions, patience and invaluable assistance at every stage of the research project. Research was funded by institutional support for the development of research organisations (IČ 47813059) by Grant SGS 16/2024 and by Support for Science and Research in the Moravian-Silesian Region 2023 (project RRC/09/2023). 

\bibliographystyle{plain}
\bibliography{resources}

\end{document}